\documentclass[a4paper,10pt,twoside]{article}

\usepackage[english]{babel}
\usepackage{amsmath,amsthm}
\usepackage{amsfonts}

\usepackage{enumerate}
\usepackage{hyperref}
\usepackage{graphicx}

\newtheorem{thmA}{Theorem}
\newtheorem{thm}{Theorem}

\newtheorem{lem}{Lemma}

\newcommand*{\di}{\, \mathrm{d} }

\begin{document}

\title{Strong approximation of Black--Scholes theory based on simple random walks}

\author{Zsolt Nika\footnote{Current address: Centrum Wiskunde \& Informatica (CWI), P.O. Box 94079
1090 GB Amsterdam, Netherlands, e-mail: nika@cwi.nl} \\
and Tam\'as Szabados\footnote{Corresponding author, address:
Department of Mathematics, Budapest University of Technology and Economics, M\H{u}egyetem rkp. 3, H
\'ep. V em. Budapest, 1521, Hungary, e-mail: szabados@math.bme.hu, telephone: (+36 1)
463-1111/ext. 5907, fax: (+36 1) 463-1677} \\
Budapest University of Technology and Economics}



\maketitle

\begin{center}
\emph{Dedicated to Endre Cs\'aki and P\'al R\'ev\'esz on the occasion of their 80th birthdays}
\end{center}

\medskip


\begin{abstract}
A basic model in financial mathematics was introduced by Black, Scholes and Merton in 1973 (BSM model). A classical discrete approximation in distribution is the binomial model given by Cox, Ross and Rubinstein in 1979 (CRR model). The BSM and the CRR models have been used for example to price European call and put options. Our aim in this work is to give a strong (almost sure, pathwise) discrete approximation of the BSM model using a suitable nested sequence of simple, symmetric random walks. The approximation extends to the stock price process, the value process, the replicating portfolio, and the greeks. An important tool in the approximation is a discrete version of the Feynman--Kac formula as well. It is hoped that such a discrete pathwise approximation can be useful for example when teaching students whose mathematical background is limited, e.g. does not contain measure theory or stochastic analysis.
\end{abstract}


\medskip

2010 \emph{MSC.} Primary 60F15. Secondary 60H30, 91G10, 97M30.

\emph{Keywords and phrases:} financial mathematics, strong approximation, random walk, discrete Feynman--Kac formula, European option.

\bigskip


\section{Introduction} \label{sec:Intro}

A basic model in financial mathematics was introduced by Black, Scholes and Merton in 1973 \cite{BS1973,Mer1973} (BSM model). Though clearly it is a rather rough model of real financial markets, its big advantage is that it can be handled by relatively simple stochastic analysis tools. A classical discrete approximation in distribution is the binomial model given by Cox, Ross and Rubinstein in 1979 \cite{CRR1979} (CRR model). The BSM and the CRR models have been used for example to price European call and put options.

Our aim in this work is to give a strong (almost sure, pathwise) discrete approximation of the BSM model using a suitable nested sequence of simple, symmetric random walks, the so-called \emph{``twist and shrink''} method, see Section \ref{sec:Pre} below. This basically means that our approximation scheme is driven by an infinite matrix $X_m(k)$ $(m=0, 1, 2, \dots; \; k=1, 2, 3, \dots)$ of independent, fair coin-tosses. The approximation extends to the stock price process, the value process, the replicating portfolio, and the greeks. An important tool in the approximation is a discrete version of the Feynman--Kac formula as well that implies the continuous version as a limit.

It is hoped that such a discrete pathwise approximation can be useful for example when teaching students whose mathematical background is limited, e.g. does not contain measure theory or stochastic analysis. For the sake of brevity, in the present paper we do not attempt to use only elementary tools; however, we think that our method can be presented in an elementary, though longer, way as well, using only the tools of a typical calculus sequence and an introductory probability course. The results of the classical BSM model rigorously follow from our elementary, discrete model by taking limits. For sake of simplicity, we limited our discussion to the case of a single risky and a single riskless assets and claims of the form $g(S(T))$, where $T > 0$ is the maturity time, $S$ is the price of the risky asset and the function $g$ satisfies suitable assumptions. We feel though that our method can be generalized to the multidimensional case and to more general claims as well. Finally, we mention that our work was partly influenced by a paper by Mikl\'os Cs\"org\H{o} \cite{Cso1999}, which had partly similar (and partly different) goals.


\section{Preliminaries of a discrete approximation: ``twist and shrink''} \label{sec:Pre}

A basic tool of the present paper is an elementary construction of Brownian motion. The specific
construction used in the sequel, taken from \cite{Szab1996}, is based on a nested sequence of simple,
symmetric random walks that uniformly converges to Brownian motion (BM = Wiener process) on bounded intervals with probability $1$. This will be called \emph{``twist and shrink''} construction. This method is a modification of the one given by Frank Knight in 1962 \cite{Kni1962} and its simplification by P\'al R\'ev\'esz in 1990 \cite{Rev1990}.

We summarize the major steps of the ``twist and shrink'' construction here. We start with \emph{a sequence of independent simple, symmetric random walks} (abbreviated: RW)
\[
S_m(0) = 0, \quad S_m(n) = \sum_{k=1}^{n} X_m(k) \quad (n \ge 1),
\]
based on an infinite matrix of independent and identically distributed random variables $X_m(k)$,
\begin{equation}\label{eq:Xmk}
\mathbb{P} \left\{ X_m(k)= \pm 1 \right\} = \frac12 \qquad (m\ge 0, k\ge 1),
\end{equation}
defined on the same complete probability space $(\Omega,\mathcal{F},\mathbb{P})$. (All stochastic processes in the sequel will be defined on this probability space.) Each random walk is a basis of an approximation of Brownian motion with a dyadic step size $\Delta t=2^{-2m}$ in time and a corresponding step size $\Delta x=2^{-m}$ in space.

The second step of the construction is \emph{twisting}. From the independent RW's we want to create
dependent ones so that after shrinking temporal and spatial step sizes, each consecutive RW becomes a
refinement of the previous one.  Since the spatial unit will be halved at each consecutive row, we
define stopping times by $T_m(0)=0$, and for $k\ge 0$,
\[
T_m(k+1)=\min \{n: n>T_m(k), |S_m(n)-S_m(T_m(k))|=2\} \qquad (m\ge 1)
\]
These are the random time instants when a RW visits even integers, different from the previous one.
After shrinking the spatial unit by half, a suitable modification of this RW will visit the same
integers in the same order as the previous RW. In other words, if $\widetilde{S}_{m-1}$ visits the integers $i_0=0, i_1, i_2, i_3, \dots$, $(i_j \ne i_{j+1})$, then we want that the twisted random walk $\widetilde{S}_m$ visit the even integers $2i_0, 2i_1, 2i_2, 2i_3$ in this order.

We operate here on each point $\omega\in\Omega$ of the
sample space separately, i.e. we fix a sample path of each RW. We define twisted RW's $\widetilde{S}_m$
recursively for $k=1,2,\dots$ using $\widetilde{S}_{m-1}$, starting with $\widetilde{S}_0(n)=S_0(n)$ $(n\ge
0)$ and $\widetilde{S}_m(0) = 0$ for any $m \ge 0$. With each fixed $m$ we proceed for $k=0,1,2,\dots$
successively, and for every $n$ in the corresponding bridge, $T_m(k)<n\le T_m(k+1)$. Each bridge is
flipped if its sign differs from the desired:
$\widetilde{X}_m(n) = \pm  X_m(n)$, depending on whether $S_m(T_m(k+1)) - S_m(T_m(k))
= 2\widetilde X_{m-1}(k+1)$ or not. So $\widetilde{S}_m(n)=\widetilde{S}_m(n-1)+\widetilde{X}_m(n)$.

Then $(\widetilde{S}_m(n))_{n\ge 0}$ is still a
simple symmetric RW \cite[Lemma 1]{Szab1996}. The twisted RW's have the desired refinement property:
\[
\widetilde{S}_{m+1}(T_{m+1}(k)) = 2 \widetilde{S}_{m}(k) \qquad (m\ge 0, k\ge 0).
\]

The third step of the RW construction is \emph{shrinking}. The sample paths of $\widetilde{S}_m(n)$ $(n\ge
0)$ can be extended to continuous functions by linear interpolation, this way one gets
$\widetilde{S}_m(t)$ $(t\ge 0)$ for real $t$. The $mth$ \emph{``twist and shrink'' RW} is defined by
\begin{equation}\label{eq:TAS}
B_m(t)=2^{-m}\widetilde{S}_m(t2^{2m}).
\end{equation}
Then the \emph{refinement property} takes the form
\begin{equation}
B_{m+1}\left(T_{m+1}(k)2^{-2(m+1)}\right) = B_m \left( k2^{-2m}\right) \qquad (m\ge
0,k\ge 0). \label{eq:refin}
\end{equation}
Note that a refinement takes the same dyadic values in the same order as the previous shrunken walk,
but there is a \emph{time lag} in general:
\begin{equation} T_{m+1}(k)2^{-2(m+1)} - k2^{-2m} \ne 0 .
\label{eq:tlag}
\end{equation}

Now let us recall an important fact from \cite{Szab1996} and \cite{SzaSze2009} about the ``twist
and shrink'' construction that will be used in the sequel.

\begin{thmA} \label{th:Wiener}
The sequence of ``twist and shrink'' random walks ˜$B_m$ uniformly converges to Brownian motion $B$ on bounded intervals, almost surely. For all $T > 0$ fixed, as $m \to \infty$,
\begin{equation}\label{eq:Twist_n_Shrink}
\sup_{0 \le t \le T} |B(t) - B_m(t)| = O\left(m^{\frac34} 2^{-\frac{m}{2}}\right) \qquad \text{a.s.}
\end{equation}
\end{thmA}

We mention that the ``twist and shrink'' approximation is asymptotically equivalent to a sequence of Skorohod's embedding of simple, symmetric random walks into Brownian motion, see \cite{Szab1996}.

\section{Summary of the Black--Scholes--Merton model} \label{sec:BSM}

Here we summarize the most important facts about the classical Black--Scholes--Merton model (BSM model) of financial mathematics. We do it in order to build a reference frame to which one may compare the results that we show in the next sections by elementary tools. At the same time, our purpose is to obtain these results from our strong discrete approximation as limits.

Let $(\Omega, \mathcal{F}, \mathbb{P})$ denote a complete probability space, with Brownian motion  $(B(t))_{t\ge0}$, $B(0)=0$, and its generated standard filtration $(\mathcal{F}_t)_{t\ge0}$. Let $a(t)$ and $b(t)$ denote two predictable processes: the number of shares (risky asset) and bond units (riskless asset) held at time $t$, respectively. The market value of this portfolio at time $t$ is
\[
V(t) := a(t) S(t) + b(t) \beta(t) ,
\]
where $S(t)$ is the price process of the risky asset, assumed to be a geometric Brownian motion, defined by
\begin{equation}\label{eq:GBM}
dS(t) = \mu S(t) dt + \sigma S(t) dB(t), \quad S(0)=s_0 > 0
\end{equation}
($\mu \in \mathbb{R}, \sigma > 0$) and $\beta(t)$ is the price of the riskless asset, defined by
\[
d\beta(t) = r(t) \beta (t) dt, \quad \beta(0) = \beta_0 > 0.
\]
For simplicity, we assume that the riskless interest rate $r(t) = r > 0$ (constant) and $\beta_0=1$; then $\beta (t) = e^{rt}$. $S(t)$ is the unique solution of the stochastic differential equation (\ref{eq:GBM}):
\begin{equation}\label{eq:S}
S(t) = s_0 \exp\left(\left(\mu - \sigma^2/2 \right)t + \sigma B(t)\right),
\end{equation}
with $\mu \in \mathbb{R}$ and $\sigma > 0$.

It is also assumed that the above portfolio is \emph{self-financing}, that is, the change of the value process comes only from the change in the prices of assets:
\begin{equation}\label{eq:self_finance}
dV(t) = a(t) dS(t) + b(t) d\beta(t).
\end{equation}

It is very important that under the above assumptions, on a fixed time interval $[0,T]]$, one can introduce an equivalent martingale measure (EMM) $\mathbb{Q}$ by
\begin{equation}\label{eq:Q}
\frac{d\mathbb{Q}}{d\mathbb{P}} := \exp\left( \frac{r - \mu}{\sigma} B(T) - \frac12 \left(\frac{r - \mu}{\sigma}\right)^2 T \right).
\end{equation}
Under the probability $\mathbb{Q}$, the discounted price process $S(t)/\beta(t) = e^{-rt} S(t)$ is a martingale, the SDE (\ref{eq:GBM}) is transformed into
\begin{equation}\label{eq:GBMQ}
dS(t) = r S(t) dt + \sigma S(t) dW(t),
\end{equation}
where
\begin{equation}\label{eq:W}
W(t):= B(t) + \frac{\mu - r}{\sigma}t
\end{equation}
is $\mathbb{Q}$-Brownian motion, and
\begin{equation}\label{eq:StQ}
S(t) = s_0 \exp\left(\left(r - \sigma^2/2 \right)t + \sigma W(t)\right).
\end{equation}

Here we discuss pricing claims of the form $g(S(T))$, where $T>0$ denotes the time of maturity of the claim and $g \in C(\mathbb{R}_+)$. For example, in the case of a \emph{European call option}, one is paid $(S(T) - K)_+ = \max(S(t) - K, 0)$ at maturity $T>0$, where $K>0$ is the strike price. So here $g(x) = (x - K)_+$; similarly, in the case of \emph{European put option}, $g(x) = (K - x)_+$. A major question is how much one has to pay for such an option at time $t \in [0, T]$. It follows from the martingale property of $V(t)/\beta(t)$ under the probability $\mathbb{Q}$ that the answer is the $\mathbb{Q}$-expectation of the discounted payment:
\begin{equation}\label{eq:ftx}
f(t,x) := \mathbb{E}_{\mathbb{Q}}\left(e^{-r(T-t)} g(S(T)) \mid S(t) = x \right)
\end{equation}
($0 \le t \le T, x > 0$).
In the case of European options, there is a \emph{put-call parity}:
\begin{equation}\label{eq:parity}
C(t,x) - P(t,x) = x - e^{-r(T-t)}K ,
\end{equation}
where $C$ and $P$ denote price of the call and put options, respectively.

It follows from (\ref{eq:StQ}) that $S(t)$ has lognormal distribution, with each parameter, except possibly the volatility $\sigma$, known. Thus one obtains an explicit expression for the price of a European call option:
\begin{equation}\label{eq:price}
C(t,x) = x \Phi(d_+(T-t,x)) -e^{-r(T-t)} K \Phi(d_-(T-t,x)),
\end{equation}
where
\begin{equation}\label{eq:price_d}
d_{\pm}(u,x) := \frac{1}{\sigma \sqrt{u}} \left(\log\left(\frac{x}{K}\right) + \left(r \pm \frac{\sigma^2}{2} \right) u \right),
\end{equation}
$0 \le t < T$, and $\Phi$ is the standard normal distribution function. The volatility $\sigma$ of a given risky asset is usually estimated from its history.

By the \emph{Feynman--Kac formula}, under suitable conditions, $f(t,x)$ defined by (\ref{eq:ftx}) uniquely solves a second order partial differential equation (PDE), the famous \emph{Black--Scholes equation}:
\begin{equation}\label{eq:BS}
\partial_t f(t,x) + r x \partial_x f(t,x) + \frac12 \sigma^2 x^2 \partial_{xx} f(t,x)  - r f(t,x) = 0 ,
\end{equation}
with the boundary condition $f(T,x) = g(x)$. A sufficient condition is that $g \in C^2_c(\mathbb{R}_+)$ (That is, $g$ is twice continuously differentiable with compact support.) Though $g(x) = (x - K)_+$ is not differentiable at the point $x = K$, however formulas (\ref{eq:price}) and (\ref{eq:price_d}) imply that $C(t,x) \in C^{1,2}([0, T) \times (0, \infty)$, it solves (\ref{eq:BS}), and $\lim_{t \to T^-} C(t,x) = (x -K)_+$.

It is also important to determine a \emph{replicating portfolio} that can guarantee that the seller of a call option can always have a portfolio of the same value as the price of the option (\emph{there is no arbitrage}):
\begin{equation}\label{eq:value}
V(t) = a(t) S(t) + b(t) \beta(t) = f(t, S(t)).
\end{equation}
Then, supposing $f \in C^{1,2}(\mathbb{R}_+ \times \mathbb{R}$, applying It\^o's formula to $f(t, S(t))$, and comparing the terms with the corresponding terms of the self-financing condition (\ref{eq:self_finance}), one obtains the replicating portfolio as
\begin{equation}\label{eq:at}
a(t) = \partial_x f(t, S(t))
\end{equation}
and
\begin{multline}\label{eq:bt}
b(t) = \frac{1}{r}e^{-rt} \left\{\partial_t f(t, S(t)) + \frac12 \sigma^2 S^2(t) \partial_{xx} f(t, S(t))\right\}\\
= e^{-rt} \left\{ f(t, S(t)) -  S(t) \partial_x f(t, S(t)) \right\} .
\end{multline}
At the second equality above we used the Black--Scholes equation (\ref{eq:BS}). The partial derivatives appearing in the above formulas are called greeks, because they usually are denoted by Greek letters.
In the case of European call option, these lead to the formulas
\begin{equation}\label{eq:eur_port}
a(t) = \Phi(d_+(T-t, S(t))) , \quad b(t) = -K e^{-rT} \Phi(d_-(T-t, S(t))),
\end{equation}
in accordance with (\ref{eq:price}) and (\ref{eq:value}).

\section{A strong discrete approximation of the BSM model} \label{sec:DBSM}

In the sequel we use the following notations when $m=0, 1, 2, \dots$ is (temporarily) fixed: $\Delta t = 2^{-2m}$, $\Delta x = 2^{-m}$, and $t_k = k \Delta t$ $(k=0,1,2,\dots)$. Beside the filtrated probability space  $(\Omega, \mathcal{F}, (\mathcal{F}_t)_{t \ge 0}, \mathbb{P})$ and Brownian motion $(B(t))_{t \ge 0}$ of the previous section we also use the ``twist and shrink'' random walk $B_m(t_k)$ of Section \ref{sec:Pre} and the discrete time filtration $(\mathcal{F}^m_{t_k})_{k \ge 0}$ generated by it, that is $\mathcal{F}^m_{t_k}$ is the smallest $\sigma$-algebra containing all events defined by $\{B_m(t_j) : j=1, \dots, k\}$.

Let $a_m(t_k)$ and $b_m(t_k)$ denote two \emph{predictable} processes: the number of shares (risky asset), respectively, the bond units (riskless asset) held at time $t_k$, more precisely, over the time period $[t_k, t_{k+1})$. Thus they should be measurable w.r.t. $\mathcal{F}^m_{t_{k}}$ for each $k \ge 1$. The market value of this portfolio at time $t_k$ is
\begin{equation}\label{eq:Vdef}
V_m(t_k) := a_m(t_k) S_m(t_k) + b_m(t_k) \beta_m(t_k) ,
\end{equation}
where $S_m(t_k)$ is the price process of the risky asset and $\beta_m(t_k)$ is the price of the riskless asset.

An important special case that we are going to use unless otherwise is stated is when the functions $\mu$ and $\sigma$ are linear:
\begin{equation}\label{eq:linear}
\Delta S_m(t_{k+1}) = \mu S_m(t_k) \Delta t + \sigma S_m(t_k) \Delta B_m(t_{k+1}), \quad S_m(t_0) = s_0 > 0 ,
\end{equation}
$\mu \in \mathbb{R}$ and $\sigma > 0$ (with a small abuse of notation). This defines a simple recursion
\begin{equation}\label{eq:linear1}
S_m(t_{k+1}) = S_m(t_k) \left\{ 1 + \mu \Delta t + \sigma \Delta B_m(t_{k+1}) \right\} \quad (k \ge 0)
\end{equation}
in terms of a symmetric coin-tossing sequence
\begin{equation}\label{eq:Xmtk}
X_m(t_{k+1}) := 2^m \Delta B_m(t_{k+1}) = \pm 1 .
\end{equation}

This time discrete process corresponds to the geometric Brownian motion (\ref{eq:GBM}) and so it has a simple approximate solution.
\begin{lem} \label{le:disc_appr}
Denoting
\[
\widetilde{S}_m(t) := s_0 \exp\left(\left(\mu - \frac{\sigma^2}{2}\right)t + \sigma B_m(t)\right) \quad (t \ge 0),
\]
we obtain for any $m \ge 0$ that
\[
\sup_{0 \le t_k \le T} \left|S_m(t_k) - \widetilde{S}_m(t_k) \right| \le c_1 2^{-m},
\]
where $S_m$ is the solution of (\ref{eq:linear1}) and $c_1 = c_1(\mu, \sigma, T) \in\mathbb{R}_+$.
\end{lem}
\begin{proof}
Using Taylor expansion of the exponential function, we get
\begin{multline*}
\Delta \widetilde{S}_m(t_{k+1}) := \widetilde{S}_m(t_{k+1}) - \widetilde{S}_m(t_k) \\
= \widetilde{S}_m(t_k) \left\{\exp\left((\mu - \sigma^2/2) \Delta t + \sigma \Delta B_m(t_{k+1})\right) - 1 \right\} \\
= \widetilde{S}_m(t_k) \left\{ (\mu - \sigma^2/2) \Delta t + \sigma \Delta B_m(t_{k+1}) + \frac12 \left( (\mu - \sigma^2/2) \Delta t + \sigma \Delta B_m(t_{k+1}) \right)^2 \right. \\
\left. + \frac{1}{3!} e^t \left( (\mu - \sigma^2/2) \Delta t + \sigma \Delta B_m(t_{k+1}) \right)^3 \right\}, \end{multline*}
where $0 < |t| < (\mu - \sigma^2/2) 2^{-2m} + \sigma 2^{-m}$. Thus
\[
\left| \Delta \widetilde{S}_m(t_{k+1})- \widetilde{S}_m(t_k) \left( \mu \Delta t + \sigma \Delta B_m(t_{k+1}) \right) \right| \le C'_1 2^{-3m} ,
\]
where $C'_1 = C'_1(\mu, \sigma) \in\mathbb{R}_+$. By induction, it implies the statement of the lemma.
\end{proof}

By this lemma and (\ref{eq:Twist_n_Shrink}) it is also clear that
\begin{equation}\label{eq:Smconv}
\sup_{0 \le t \le T} \left|S_m(t) - S(t) \right| = O\left(m^{\frac34} 2^{-\frac{m}{2}}\right) \qquad \text{a.s.}
\end{equation}
for any $T > 0$, where $S$ is defined by (\ref{eq:S}) and $S_m$ is linearly interpolated between the time instants $t_k$.

Also, the price $\beta_m(t_k)$ of the riskless asset can be defined as
\[
\Delta \beta_m(t_{k+1}) := \beta_m(t_{k+1}) - \beta_m(t_{k}) := r(t_k) \beta_m(t_k) \Delta t, \quad \beta_m(0) = \beta_0 > 0.
\]
For simplicity, we assume that the riskless interest rate $r_m(t_k) = r > 0$ (constant) and $\beta_0=1$; then \begin{equation}\label{eq:beta_m}
\beta_m(t_k) = (1+r \Delta t)^k \quad \text{and} \quad \beta_m(t) = (1+r \Delta t)^{\lfloor t/\Delta t \rfloor} \quad (t \ge 0).
\end{equation}
Then for any $m \ge 0$,
\begin{equation}\label{eq:betamconv}
\sup_{0 \le t \le T} \left|\beta(t) - \beta_m(t) \right| = \sup_{0 \le t \le T} \left|e^{rt} - (1+r 2^{-2m})^{\lfloor t 2^{2m} \rfloor} \right|\le  c_2 2^{-2m},
\end{equation}
where $c_2 = c_2(r) \in\mathbb{R}_+$.

We assume that the portfolio is \emph{self-financing}, that is, the change of the value process comes only from the change in the prices of assets over any time interval $[t_k, t_{k+1}]$:
\begin{equation}\label{eq:self_fin}
\Delta V_m(t_{k+1}) := V_m(t_{k+1}) - V_m(t_{k}) = a_m(t_{k}) \Delta S_m(t_{k+1}) + b_m(t_{k}) \Delta \beta_m(t_{k+1}).
\end{equation}
Since
\begin{multline*}
\Delta V_m(t_{k+1}) \\
= a_m(t_{k+1}) S_m(t_{k+1}) + b_m(t_{k+1}) \beta_m(t_{k+1}) - \left\{ a_m(t_{k}) S_m(t_{k}) + b_m(t_{k}) \beta_m(t_{k}) \right\} \\
= a_m(t_{k}) \Delta S_m(t_{k+1}) + b_m(t_{k}) \Delta \beta_m(t_{k+1}) \\
+ \Delta a_m(t_{k+1}) S_m(t_{k+1}) + \Delta b_m(t_{k+1}) \beta_m(t_{k+1}),
\end{multline*}
the self-financing condition (\ref{eq:self_fin}) is equivalent to the equality
\begin{equation}\label{eq:self_fin2}
a_m(t_{k+1}) S_m(t_{k+1}) + b_m(t_{k+1}) \beta_m(t_{k+1}) = a_m(t_{k}) S_m(t_{k+1}) + b_m(t_{k}) \beta_m(t_{k}).
\end{equation}
That is, at a new time instant $t_{k+1}$, the new portfolio $(a_m(t_{k+1}), b_m(t_{k+1}))$ must have the same value as the previous one $(a_m(t_{k}), b_m(t_{k}))$.

It is very important that on a fixed time interval $[0, T]$, where $T$ is an integer multiple of $\Delta t = 2^{-2m}$, one can introduce an \emph{equivalent martingale measure} or \emph{risk-neutral probability}  $\mathbb{Q}_m$ \emph{on the time discrete filtration} $(\mathcal{F}^m_{t_k})_{k \ge 0}$. Based on (\ref{eq:beta_m}) and (\ref{eq:linear1}), set
\begin{eqnarray}\label{eq:um_dm}
  r_m &:=& 1 + r \Delta t = 1 + r 2^{-2m}, \nonumber \\
  u_m &:=& 1 + \sigma 2^{-m} +  \mu 2^{-2m}, \nonumber \\
  d_m &:=& 1 - \sigma 2^{-m} +  \mu 2^{-2m} ,
\end{eqnarray}
and the risk neutral probabilities $q_m^+$ of an up-step (tossing head) and $q_m^-$ of a down-step (tossing tail) by
\begin{eqnarray}\label{eq:qmqm}
q_m^+ &=& \frac{r_m - d_m}{u_m - d_m} =  \frac12 + \frac12 \frac{r - \mu}{\sigma} 2^{-m},  \nonumber \\
q_m^- &=& 1 - q_m^+ = \frac{u_m - r_m}{u_m - d_m} =   \frac12 - \frac12 \frac{r - \mu}{\sigma} 2^{-m}.
\end{eqnarray}
From now on we always assume that $0 < q_m^+ < 1$: this certainly holds when $m$ is large enough,
\begin{equation}\label{eq:qm_positive}
m > m_0 := \frac{1}{\log 2} \log\left(\frac{|r-\mu|}{\sigma}\right) .
\end{equation}
The value of the ratio $S_m(t_{k+1})/S_m(t_k)$ is either $u_m$ with risk-neutral probability $q_m^+$, or $d_m$ with risk-neutral probability $q_m^-$. Then define the probability measure $\mathbb{Q}_m$ by how the symmetric $\frac12$-$\frac12$ probabilities change to $q_m^+$-$q_m^-$, that is, with its Radon--Nikodym derivative on the filtration $(\mathcal{F}^m_{t_k})_{k \ge 0}$:
\begin{multline}\label{eq:Qm}
\frac{d \mathbb{Q}_m}{d \mathbb{P}} :=  \left(\frac{q_m^+}{\frac12}\right)^{\#\text{Heads}(T)} \,\left(\frac{q_m^-}{\frac12}\right)^{\#\text{Tails}(T)} \\
= (2 q_m^+)^{\frac12(T 2^{2m} + B_m(T) 2^{m})} \, (2 q_m^-)^{\frac12(T 2^{2m} - B_m(T) 2^{m})} \\
= \exp\left\{\frac{r - \mu}{\sigma} B_m(T) - \frac12 \left(\frac{r - \mu}{\sigma}\right)^2 T + O(2^{-m}) \right\} ,
\end{multline}
where $\#\text{Heads}(T)$ and $\#\text{Tails}(T)$ denote the number of heads (+1's), respectively, tails (-1's) in the sequence (\ref{eq:Xmtk}) as $t_k$ runs from $0$ to $T$. In (\ref{eq:Qm}) we used that
\[
\log(2 q_m^{\pm}) = \log \left(1 \pm \frac{r-\mu}{\sigma} 2^{-m}\right) = \pm \frac{r-\mu}{\sigma} 2^{-m} - \frac12 \left(\frac{r - \mu}{\sigma}\right)^2 2^{-2m} + O(2^{-3m}) .
\]
Compare the last formula of (\ref{eq:Qm}) to (\ref{eq:Q}).

Here we mention some facts about the probability $\mathbb{Q}_m$.
\begin{lem}\label{le:Qmprop}
(a) The process
\[
\Lambda(t_k) := (2 q_m^+)^{\frac12(t_k 2^{2m} + B_m(t_k) 2^{m})} \, (2 q_m^-)^{\frac12(t_k 2^{2m} - B_m(t_k) 2^{m})} \quad (0 \le t_k \le T)
\]
is a positive $\mathbb{P}$-martingale w.r.t. $(\mathcal{F}^m_{t_k})_{k \ge 0}$, with expectation 1.

(b) For the total variation distance between the probabilities $\mathbb{Q}_m$ and $\mathbb{Q}$, we have
\begin{equation}\label{eq:QmQ}
\lim_{m \to \infty}\sup_{A \in \mathcal{F}} |\mathbb{Q}_m(A) - \mathbb{Q}(A)| = 0.
\end{equation}

(c) If a new nearest neighbor random walk is defined as the ``twist and shrink'' random walk $B_m$ plus a suitable drift:
\begin{equation}\label{eq:Wm}
W_m(t_k) := B_m(t_k) + \frac{\mu - r}{\sigma} t_k  ,
\end{equation}
then $W_m(t_k)$ is a $\mathbb{Q}_m$-martingale w.r.t. $(\mathcal{F}^m_{t_k})_{k \ge 0}$.

(d) Extending $W_m$ by linear interpolation to arbitrary $t \in \mathbb{R}_+$, for any $T > 0$ we have
\begin{equation}\label{eq:W_conv}
\sup_{0 \le t \le T} |W(t) - W_m(t)| = O\left(m^{\frac34} 2^{-\frac{m}{2}}\right) \qquad \text{a.s.},
\end{equation}
where $W$ is defined by (\ref{eq:W}).
\end{lem}
\begin{proof}
(a)
\begin{multline*}
\mathbb{E}_{\mathbb{P}}(\Lambda_m(t_{k+1}) \mid  \mathcal{F}^m_{t_k})  \\
= \Lambda_m(t_k) \, \mathbb{E}_{\mathbb{P}}\left((2 q_m^+)^{\frac{1 + X_m(t_{k+1})}{2}} \, (2 q_m^-)^{\frac{1 - X_m(t_{k+1})}{2}}\right) = \Lambda_m(t_k),
\end{multline*}
where $X_m(t_{k+1})$ is defined by (\ref{eq:Xmtk}).

(b) By Scheff\'e's theorem, see e.g. \cite[p. 224]{Bil1968}, it is enough to show that the Radon--Nikodym derivatives $d \mathbb{Q}_m/d \mathbb{P} = \Lambda_m(T)$ converge to $d \mathbb{Q}/d \mathbb{P} = \Lambda(T)$ $\mathbb{P}$-a.s. By (\ref{eq:Qm}),
\[
\frac{d \mathbb{Q}_m}{d \mathbb{P}} = \exp\left\{\frac{r - \mu}{\sigma} B_m(T) - \frac12 \left(\frac{r - \mu}{\sigma}\right)^2 T + O(2^{-m}) \right\} .
\]
Hence (\ref{eq:Twist_n_Shrink}) and (\ref{eq:Q}) imply the statement.

(c) Clearly, it is enough to show that $\Lambda_m(t_k) W_m(t_k)$ is a $\mathbb{P}$-martingale w.r.t. $(\mathcal{F}^m_{t_k})_{k \ge 0}$:
\begin{multline*}
\mathbb{E}_{\mathbb{P}}(\Lambda_m(t_{k+1}) W_m(t_{k+1}) \mid \mathcal{F}^m_{t_k}) \\
= \mathbb{E}_{\mathbb{P}}\left\{\left(W_m(t_k) + X_m(t_{k+1}) 2^{-m} + \frac{\mu-r}{\sigma} 2^{-2m}\right) \right. \\
\times \left. \left. \Lambda_m(t_k) (2 q_m^+)^{\frac{1 + X_m(t_{k+1})}{2}} \, (2 q_m^-)^{\frac{1 - X_m(t_{k+1})}{2}} \right| \mathcal{F}^m_{t_k} \right\} \\
= \Lambda_m(t_k) W_m(t_k) + \Lambda_m(t_k) \\
\times \mathbb{E}_{\mathbb{P}}\left\{\left(X_m(t_{k+1}) 2^{-m} + \frac{\mu-r}{\sigma} 2^{-2m}\right) (2 q_m^+)^{\frac{1 + X_m(t_{k+1})}{2}} \, (2 q_m^-)^{\frac{1 - X_m(t_{k+1})}{2}} \right\} \\
= \Lambda_m(t_k) W_m(t_k) .
\end{multline*}

(d) This statement follows from formulas (\ref{eq:Twist_n_Shrink}), (\ref{eq:W}), and (\ref{eq:Wm}).
\end{proof}

Thus, considering the measure $\mathbb{Q}_m$, the difference equation (\ref{eq:linear}) for the price process  of the risky assets can be transformed into
\begin{equation}\label{eq:Qmlinear}
\Delta S_m(t_{k+1}) = r S_m(t_k) \Delta t + \sigma S_m(t_k) \Delta W_m(t_{k+1}) .
\end{equation}

\begin{lem} \label{le:disc_mart}
The discounted price process $r_m^{-k} S_m(t_k)$ and discounted value process $r_m^{-k} V_m(t_k)$ with a self-financing portfolio $(a_m(t_k), b_m(t_k))$ are $\mathbb{Q}_m$-martingales.
\end{lem}
\begin{proof}
Using (\ref{eq:Qmlinear}) and the fact that $W_m$ is a $\mathbb{Q}_m$-martingale, one obtains
\begin{equation}\label{eq:Sm_Qm_mart}
\mathbb{E}_{\mathbb{Q}_m} \left( \left. r_m^{-k-1} S_m(t_{k+1}) \right| \mathcal{F}^m_{t_k} \right) =  r_m^{-k} S_m(t_{k}) .
\end{equation}
Then (since $\beta_m(t_k) = r_m^k$), we have
\begin{multline*}
\mathbb{E}_{\mathbb{Q}_m} \left( \left. r_m^{-k-1} V_m(t_{k+1}) \right| \mathcal{F}^m_{t_k} \right) \\
= r_m^{-k-1} \mathbb{E}_{\mathbb{Q}_m} \left( \left. a_m(t_{k+1}) S_m(t_{k+1}) + b_m(t_{k+1}) r_m^{k+1}\right| \mathcal{F}^m_{t_k} \right) \\
= r_m^{-k-1} \mathbb{E}_{\mathbb{Q}_m} \left( \left. a_m(t_{k}) S_m(t_{k+1}) + b_m(t_{k}) r_m^{k+1} \right| \mathcal{F}^m_{t_k} \right) \\
= r_m^{-k} a_m(t_{k}) S_m(t_{k}) + b_m(t_{k})
= r_m^{-k} V_m(t_{k}) .
\end{multline*}
Here we used first the self financing condition (\ref{eq:self_fin2}), then the assumption that the portfolio $(a_m(t_k), b_m(t_k))$ is predictable and equality (\ref{eq:Sm_Qm_mart}).
\end{proof}

Now we are considering the \emph{no-arbitrage} price of a claim $g(S_m(T))$ at maturity $T > 0$, where $g \in C(\mathbb{R})$. For example, we are interested in the European call option, when $g(x) = (x-K)_+$. ``No arbitrage'' means that one cannot make a profit out of nothing, with positive probability, without taking a risk. In the current situation it means that the price $f_m(t_k)$ of an option $g(S_m(T))$ at any moment $t_k \in [0, T]$ should agree with the value $V_m(t_k)$ of a \emph{self-financing portfolio replicating the claim}. The last expression means that the value of the self-financing portfolio at maturity $T$ must be equal to $g(S_m(T))$. (The details about a self-financing portfolio see in Sections \ref{sec:Port} and \ref{sec:Port_Eur} below.)
\begin{lem} \label{le:claim_price}
The arbitrage-free price of an option at time $t_k \in[0, T]$ buying a claim $g(S_m(T))$ with maturity $T = N \Delta t$ is
\[
f_m(t_k) = V_m(t_k) = \mathbb{E}_{\mathbb{Q}_m}\left(\left. r_m^{k-N} g(S_m(T)) \right| \mathcal{F}^m_{t_k} \right) ,
\]
where $V_m(t_k)$ is the value at $t_k$ of a self-financing portfolio replicating the claim.
\end{lem}
\begin{proof}
By Lemma \ref{le:disc_mart}, $r_m^{-k} V_m(t_k)$ is a $\mathbb{Q}_m$-martingale. Hence
\[
\mathbb{E}_{\mathbb{Q}_m}\left(\left. r_m^{-N} g(S_m(T)) \right| \mathcal{F}^m_{t_k} \right)
= \mathbb{E}_{\mathbb{Q}_m}\left(\left. r_m^{-N} V_m(T) \right| \mathcal{F}^m_{t_k} \right)
= r_m^{-k} V_m(t_k) ,
\]
and this is equivalent to the claim of the lemma.
\end{proof}

Lemma \ref{le:claim_price} leads to an explicit evaluation of claims $g(S_m(T))$ in the strong discrete model. At each step of time $\Delta t$, the current value of the stock price  is multiplied either by $u_m$ (up-step) or by $d_m$ (down-step), with probability $q_m^+$, respectively $q_m^-$, and these steps are independent. Thus by Lemma \ref{le:claim_price}, since $S_m$ is a Markov chain,
\begin{multline} \label{eq:fm_expl}
f_m(t_k, x) := r_m^{k-N} \mathbb{E}_{\mathbb{Q}_m} \left( \left. g(S_m(T)) \right| S_m(t_k) = x \right) \\
= r_m^{k-N} \sum_{i=0}^{N-k} \binom{N-k}{i} (q_m^+)^i (q_m^-)^{N-k-i} g(x u_m^i d_m^{N-k-i}) .
\end{multline}
E.g. for the European call option, where $g(x) = (x-K)_+$, one obtains
\begin{multline} \label{eq:Cm_expl}
C_m(t_k, x) := r_m^{k-N} \mathbb{E}_{\mathbb{Q}_m} \left( \left. (S_m(T) - K)_+ \right| S_m(t_k) = x \right) \\
= r_m^{k-N} \sum_{i=0}^{N-k} \binom{N-k}{i} (q_m^+)^i (q_m^-)^{N-k-i} (x u_m^i d_m^{N-k-i} - K)_+ \\
= r_m^{k-N} \sum_{i=j_{m,k}}^{N-k} \binom{N-k}{i} (q_m^+)^i (q_m^-)^{N-k-i} (x u_m^i d_m^{N-k-i} - K) \\
= x \, \text{Bin}(j_{m,k}; N-k, \tilde{q}_m^+) - r_m^{k-N} K \,\text{Bin}(j_{m,k}; N-k, q_m^+) ,
\end{multline}
where
\begin{equation}\label{eq:jm}
j_{m,k} := \left\lceil \frac{\log(K/x) - (N-k) \log(d_m)}{\log(u_m/d_m)}\right\rceil, \quad  \text{Bin}(j; n, p) := \sum_{i=j}^n \binom{n}{i} p^i (1-p)^{n-i},
\end{equation}
and
\[
\tilde{q}_m^+ := \frac{u_m}{r_m} q_m^+, \quad  \tilde{q}_m^- := 1 -  \tilde{q}_m^+ = \frac{d_m}{r_m} q_m^- .
\]
Not surprisingly, the result agrees with the classical binomial formula obtained by the CRR model \cite{CRR1979}.

One can get a similar formula for the price of a European put option in the discrete model, that is, for
\[
P_m(t_k, x) = r_m^{k-N} \mathbb{E}_{\mathbb{Q}_m} \left( \left. (K - S_m(T))_+ \right| S_m(t_k) = x \right) .
\]
Another way is to use a discrete version of the \emph{put-call parity}, using the simple identity $(z)_+ - (-z)_+ = z$:
\begin{equation}\label{eq:dparity}
C_m(t_k, x) - P_m(t_k, x) =  r_m^{k-N} \mathbb{E}_{\mathbb{Q}_m} \left( \left. S_m(T) - K \right| S_m(t_k) = x \right) = x - r_m^{k-N} K,
\end{equation}
since the discounted price process is a $\mathbb{Q}_m$-martingale by Lemma \ref{le:disc_mart}.

\begin{thm}\label{th:slimit}
Suppose that $g \in C_c(\mathbb{R}_+)$ and $T > 0$. As $m \to \infty$, the price $f_m(t^{(m)}, x)$ of the option $g(S_m(T))$ obtained by the above strong discrete approximation converges to its value $f(t, x)$ obtained by the Black--Scholes model, uniformly for $t \in [0, T]$ and $x > 0$:
\begin{multline*}
\lim_{m \to \infty} f_m(t^{(m)}, x) = \lim_{m \to \infty} \mathbb{E}_{\mathbb{Q}_m}\left( \left. r_m^{\lfloor t2^{2m} \rfloor - \lfloor T2^{2m}\rfloor} g(S_m(T^{(m)})) \right| S_m(t^{(m)}) = x \right ) \\
=  \mathbb{E}_{\mathbb{Q}}\left( \left. e^{-r(T-t)} g(S(T)) \right| S(t) = x \right) = f(t,x),
\end{multline*}
where $t^{(m)} := \lfloor t 2^{2m} \rfloor 2^{-2m}$.
\end{thm}
\begin{proof}
By (\ref{eq:Smconv}), $S_m$ a.s. uniformly converges to $S$ on $[0, T]$.  Since $S$ and $S_m$ are time-homogeneous Markov processes, using Lemma \ref{le:Qmprop}(b) one obtains that
\[
\lim_{m \to \infty} \sup_{t \in [0, T], x > 0} \left|\mathbb{E}^x_{\mathbb{Q}_m}\left(g(S_m(T^{(m)} - t^{(m)}))\right) - \mathbb{E}^x_{\mathbb{Q}}\left(g(S(T-t))\right)\right| = 0.
\]
This and (\ref{eq:betamconv}) prove the theorem.
\end{proof}

Because $(x - K)_+ \in C_c(\mathbb{R}_+)$, the above theorem implies that we have convergence in the case of the European put option, and then, by the put-call parity (\ref{eq:dparity}) we have convergence in the case of the European call option as well:
\[
\lim_{m \to \infty} P_m(t, x) = P(t,x), \quad \lim_{m \to \infty} C_m(t, x) = C(t,x).
\]

Naturally, the convergence of the explicit price formulas above follow from this, but it may be instructive to see how this convergence follows from the DeMoivre--Laplace theorem in an elementary way as well. As $m \to \infty$, the right hand side of (\ref{eq:Cm_expl}) and formulas in (\ref{eq:jm}) tend to the corresponding expressions (\ref{eq:price}) and (\ref{eq:price_d}) of the Black--Scholes theory. For, by (\ref{eq:um_dm}), ignoring the smaller order terms,
\[
\log\left(\frac{u_m}{d_m}\right) = \log\left( \frac{1 + \sigma 2^{-m} + \mu 2^{-2m}}{1 - \sigma 2^{-m} + \mu 2^{-2m}} \right) \sim 2 \sigma 2^{-m},
\]
\[
\log(d_m) = \log(1 - \sigma 2^{-m} + \mu 2^{-2m}) \sim \left((\mu - \sigma^2/2) 2^{-2m} - \sigma 2^{-m} \right),
\]
and standardizing the binomial distributions:
\begin{multline*}
\frac{j_m - N q_m^+}{\sqrt{N q_m^+ \, q_m^-}} \\
\sim  \frac{\log(\frac{K}{s_0}) - N \left((\mu - \frac{\sigma^2}{2}) 2^{-2m} - \sigma 2^{-m} \right) - 2 \sigma 2^{-m} N \left(\frac12 - \frac12 \frac{r - \mu}{\sigma} 2^{-m}\right)}{2 \sigma 2^{-m} \sqrt{N \left(\frac12 - \frac12 \frac{r - \mu}{\sigma} 2^{-m}\right)\left(\frac12 + \frac12 \frac{r - \mu}{\sigma} 2^{-m}\right)}} \\
\sim \frac{\log(\frac{K}{s_0}) - (r - \frac{\sigma^2}{2}) T }{\sigma \sqrt{T}},
\end{multline*}
\begin{multline*}
\frac{j_m - N \tilde{q}_m^+}{\sqrt{N \tilde{q}_m^+ \tilde{q}_m^-}} \\
\sim  \frac{\log(\frac{K}{s_0}) - N \left((\mu - \frac{\sigma^2}{2}) 2^{-2m} - \sigma 2^{-m} \right) -  \sigma 2^{-m} N \left(1 - \frac{r - \mu}{\sigma} 2^{-m}\right) \left(1 + \sigma 2^{-m} \right)}{2 \sigma 2^{-m} \sqrt{N \frac14}} \\
\sim \frac{\log(\frac{K}{s_0}) - T \left((\mu - \frac{\sigma^2}{2}) + r - \mu + \sigma^2\right)}{\sigma \sqrt{T}}
\sim \frac{\log(\frac{K}{s_0}) - (r + \frac{\sigma^2}{2}) T}{\sigma \sqrt{T}}.
\end{multline*}
This way, we recover (\ref{eq:price}) and (\ref{eq:price_d}).

\section{Replicating portfolio} \label{sec:Port}

Here we want to deduce a replicating self-financing portfolio of a claim $g(S_m(T))$, where $g \in C(\mathbb{R})$, in the discrete setting discussed above. For a fixed value of $m$ it agrees with the classical solution and can be found in the literature, see e.g. \cite{Shr2004}. Our goal here is to find approximations when $m$ is large enough and to show that the limits when $m \to \infty$ are the classical Black--Scholes ones.

At the beginning, we fix an $m \ge m_0$, where $m_0$ is defined by (\ref{eq:qm_positive}). Recall that by (\ref{eq:Vdef}), the value of a portfolio at time $t_k$ is
\[
V_m(t_k) = a_m(t_k) S_m(t_k) + b_m(t_k) r_m^{k} .
\]
By the self-financing equality (\ref{eq:self_fin}), at time $t_{k+1}$ its value becomes
\begin{multline}\label{eq:Vmtk1}
V_m(t_{k+1}) = a_m(t_k) S_m(t_{k+1}) + b_m(t_k) r_m^{k+1} \\
= a_m(t_k) S_m(t_{k+1}) + r_m \left( V_m(t_k) - a_m(t_k) S_m(t_k)\right)
\end{multline}
Equivalently,
\begin{equation}\label{eq:value_back}
r_m^{-k} V_m(t_k) = r_m^{-k-1} V_m(t_{k+1}) - a_m(t_k) \left(r_m^{-k-1} S_m(t_{k+1}) -  r_m^{-k} S_m(t_k) \right).
\end{equation}
If $S_m(t_k)$ is given, then -- as we saw in (\ref{eq:um_dm}) and (\ref{eq:qmqm}) -- $S_m(t_{k+1})$ can have two possible values: $u_m S_m(t_k)$ or $d_m S_m(t_k)$, with risk neutral probabilities $q_m^+$ and $q_m^-$, respectively. So imagine two equations replacing $S_m(t_{k+1})$ in (\ref{eq:value_back}) by these two possible values. Then multiply them by their corresponding risk-neutral probabilities and sum. The result -- explicitly denoting the given $x=S_m(t_k)$ -- is
\begin{equation}\label{eq:Vmtk}
V_m(t_k, x) = r_m^{-1} \left(q_m^+ V_m(t_{k+1}, u_m x) + q_m^- V_m(t_{k+1}, d_m x)\right) .
\end{equation}
This is exactly the same as the statement in Lemma \ref{le:disc_mart} that $r_m^{-k} V_m(t_k)$ is a $\mathbb{Q}_m$-martingale.

Equation (\ref{eq:Vmtk}) makes it possible to recursively determine the value of $V_m(t_k, x)$ for any $t_k \in[0, T]$ and $x = S_m(t_k)$. Assuming that $T = N \Delta t$, we start with the boundary value $g(S_m(T))$ of the claim. Then take the time $t_{N-1}$ and with all possible values $x = S_m(t_{N-1})$ set
\[
V_m(t_N, u_m x ) = g(u_m x), \quad V_m(t_N, d_m x) = g(d_m x)
\]
to determine $V_m(t_{N-1}, x)$. Then we proceed backward recursively with time steps $-\Delta t$, until time $0$. By Lemma \ref{le:claim_price} and (\ref{eq:Vmtk}) above, this exactly corresponds to the no arbitrage pricing of the claim, based on the $\mathbb{Q}_m$-martingale property of $r_m^{-k} V_m(t_k)$. In other words, at the end of the recursion we have
\begin{equation}\label{eq:fm_Vm}
f_m(t_k, x) = V_m(t_k, x)
\end{equation}
for each $t_k \in [0, T]$ and each possible value $x = S_m(t_k)$. Still, we have to check that it is possible to define a self-financing replicating portfolio that corresponds to the no arbitrage value process $V_m$ above.

Again, starting with (\ref{eq:Vmtk1}), and considering a possible up-step and down-step, we get a $2 \times 2$ linear system for the portfolio:
\begin{eqnarray*}
a_m(t_k) u_m S_m(t_k) + b_m(t_k) r_m^{k+1} &=& V_m(t_{k+1}, u_m S_m(t_k)) \\
a_m(t_k) d_m S_m(t_k) + b_m(t_k) r_m^{k+1} &=& V_m(t_{k+1}, d_m S_m(t_k)) .
\end{eqnarray*}
Since $d_m  < u_m$, the above system is uniquely solvable:
\begin{equation}\label{eq:amtk}
a_m(t_k) = \frac{\Delta V_m(t_{k+1})}{\Delta S_m(t_{k+1})} := \frac{V_m(t_{k+1}, u_m S_m(t_k)) - V_m(t_{k+1}, d_m S_m(t_k))}{u_m S_m(t_k) - d_m S_m(t_k)},
\end{equation}
and
\begin{equation}\label{eq:bmtk}
b_m(t_k) = r_m^{-k-1} \left(V_m(t_{k+1}, u_m S_m(t_k)) - a_m(t_k) u_m S_m(t_k) \right).
\end{equation}
It is clear from these that the portfolio $(a_m(t_k), b_m(t_k))$, which is used on the time interval $[t_k, t_{k+1})$, is measurable w.r.t. $\mathcal{F}^m_{t_k}$, so is predictable. In fact, $(a_m(t_k), b_m(t_k))$ depends only on the present value $S_m(t_k)$ of the price of the risky asset at time $t_k$. Thus whenever we need a specific value of the portfolio in the case $x=S_m(t_k)$ is given, we are going to use the notation $(a_m(t_k, x), b_m(t_k, x))$. Remember that one can write $f_m$ instead of $V_m$ everywhere in the previous formulas.

We need the following lemma to find derivatives of $f_m = V_m$.
\begin{lem} \label{le:diff_fm}
If $g \in C^\ell(\mathbb{R})$, then $f_m(t_k, x)$ is $\ell$-times continuously differentiable with respect to $x > 0$, for any $t_k \in [0, T]$.
\end{lem}
\begin{proof}
Using (\ref{eq:fm_expl}), we see that for any $0 \le j \le \ell$,
\begin{equation}\label{eq:jth_der}
(\partial_x)^j f_m(t_k, x)
= r_m^{k-N} \sum_{i=0}^{N-k} \binom{N-k}{i} (q_m^+)^i \, (q_m^-)^{N-k-i} \left(\frac{s_i}{x}\right)^j g^{(j)}(s_i) ,
\end{equation}
where $s_i := x u_m^i d_m^{N-k-i}$.
\end{proof}

This lemma has an analogue in the continuous case as well.
\begin{lem} \label{le:diff_f}
If $g \in C_c^\ell(\mathbb{R})$, then $f(t_k, x)$ defined by (\ref{eq:ftx}) is $\ell$-times continuously differentiable with respect to $x > 0$, for any $t \in [0, T]$.
\end{lem}
\begin{proof}
Since $W$ is $\mathbb{Q}$-Brownian motion and $S$ defined by (\ref{eq:StQ}) is Markov, we obtain that
\begin{multline*}
f(t, x) = e^{-r(T-t)} \mathbb{E}^x_{\mathbb{Q}}\left( g(S(T-t)) \right) \\
= e^{-r(T-t)} \int_{-\infty}^{\infty} g\left(x e^{\left(r - \frac{\sigma^2}{2} \right)(T-t) + \sigma y } \right) \frac{1}{\sqrt{2 \pi (T-t)}} e^{-\frac{(y-x)^2}{2(T-t)}} \di y .
\end{multline*}
Since we supposed that $g$ has compact support, the differentiation and the integration can be interchanged. Thus we see that for any $0 \le j \le \ell$,
\[
(\partial_x)^j f(t, x)
= e^{-r(T-t)} \int_{-\infty}^{\infty} \left(\frac{s(y)}{x}\right)^j g^{(j)}\left(s(y)\right) \frac{1}{\sqrt{2 \pi (T-t)}} e^{-\frac{(y-x)^2}{2(T-t)}} \di y
\]
where $s(y) := x e^{\left(r - \frac{\sigma^2}{2} \right)(T-t) + \sigma y } $.
\end{proof}

\begin{thm}\label{th:appr_port}
Suppose that $g \in C^2_c(\mathbb{R}_+)$ and
\begin{equation}\label{eq:Maxgdprime}
M := \sup_{s \ge 0} |g''(s)| < \infty .
\end{equation}

(a) Then there exists a constant $c_3 = c_3(\mu, \sigma, r, T) \in \mathbb{R}_+$, such that for any $m > m_0$, $t_k \in [0, T]$, and any value $x = S_m(t_k)$ we have
\[
\left| a_m(t_k, x) - \partial_x f_m(t_{k+1}, x)\right| \le c_3 M x \, 2^{-m},
\]
\begin{multline*}
\left| b_m(t_k, x) - r_m^{-k-1} \left(f_m(t_{k+1}, u_m x) -  u_m x \, \partial_x f_m(t_{k+1}, x) \right) \right| \\
\le c_3 M (1 + \sigma 2^{-m} +|\mu| 2^{-2m}) x^2 \, 2^{-m}.
\end{multline*}

(b) Moreover, using the notation $t^{(m)} = \lfloor t 2^{2m} \rfloor 2^{-2m}$,
\[
\lim_{m \to \infty} \sup_{0\le t \le T} |a_m(t^{(m)}) - a(t)| = 0 , \quad \lim_{m \to \infty} \sup_{0\le t \le T} |b_m(t^{(m)}) - b(t)| = 0 \quad \text{a.s.},
\]
where the portfolio $(a(t), b(t))$ of the Black--Scholes model is defined by (\ref{eq:at}) and (\ref{eq:bt}).
\end{thm}
\begin{proof}
(a) Formulas (\ref{eq:fm_Vm}) and (\ref{eq:amtk}) show that
\[
a_m(t_k, x) = \frac{f_m(t_{k+1}, u_m x) - f_m(t_{k+1}, d_m x)}{(u_m - d_m) x},
\]
where $x$ is a possible value of $S_m(t_k)$. The mean value theorem gives that $a_m(t_k, x) = \partial_x f_m(t_k, \xi)$ for some point $\xi \in (d_m x, u_m x)$. By (\ref{eq:um_dm}), $u_m - d_m = 2 \sigma 2^{-m}$ and $(u_m + d_m)/2 = 1 + \mu 2^{-2m}$; thus
\begin{multline*}
|\xi - x| \le |\xi - x(u_m + d_m)/2| + |x(u_m + d_m)/2 - x| \le \sigma x \, 2^{-m} + |\mu| x \, 2^{-2m} \\
\le (\sigma + |\mu|2^{-m}) x \,2^{-m} .
\end{multline*}
By Lemma \ref{le:diff_fm},
\begin{multline*}
\left| \partial_x f_m(t_{k+1}, \xi) - \partial_x f_m(t_{k+1}, x)\right| \\
\le r_m^{k+1-N} \sum_{i=0}^{N-k-1} \binom{N-k-1}{i} (q_m^+ u_m)^i \, (q_m^- d_m )^{N-k-1-i} \\
\times \left|g'(\xi u_m^i d_m^{N-k-1-i}) - g'(x u_m^i d_m^{N-k-1-i})\right|
\end{multline*}
The assumption  $g \in C^2_c(\mathbb{R}_+)$ implies that
\[
\left|g'(\xi u_m^i d_m^{N-k-1-i}) - g'(x u_m^i d_m^{N-k-1-i})\right| \le M |\xi - x| u_m^i d_m^{N-k-1-i},
\]
where $M$ is defined by (\ref{eq:Maxgdprime}).

Combining the formulas above, it follows that
\begin{multline*}
\left| \partial_x f_m(t_{k+1}, \xi) - \partial_x f_m(t_{k+1}, x)\right| \\
\le M |\xi - x| \sum_{i=0}^{N-k-1} \binom{N-k-1}{i} \left(\frac{q_m^+ u_m^2}{r_m}\right)^i \, \left(\frac{q_m^- d_m^2}{r_m} \right)^{N-k-1-i} \\
\le (\sigma + |\mu|2^{-m}) M x \, 2^{-m} \left(\frac{q_m^+ u_m^2}{r_m} + \frac{q_m^- d_m^2}{r_m}\right)^{N-k-1}.
\end{multline*}

Now
\begin{multline*}
\frac{q_m^+ u_m^2}{r_m} + \frac{q_m^- d_m^2}{r_m} = \frac{u_m^2 (r_m - d_m) + d_m^2 (u_m - r_m)}{r_m (u_m - d_m)} = u_m + d_m - \frac{u_m d_m}{r_m} \\
= 2(1 + \mu 2^{-2m}) - \frac{(1 + \sigma 2^{-m} + \mu 2^{-2m})(1 - \sigma 2^{-m} + \mu 2^{-2m})}{1 + r 2^{-2m}}\\
= 1 +(\sigma^2 + r)2^{-2m} + O(2^{-3m}).
\end{multline*}
Hence
\begin{multline*}
\left|a_m(t_k, x) - \partial_x f_m(t_{k+1}, x)\right|  = \left| \partial_x f_m(t_{k+1}, \xi) - \partial_x f_m(t_{k+1}, x)\right| \\
\le (\sigma + |\mu|2^{-m}) M x \, 2^{-m} \exp\left((\sigma^2 + r + O(2^{-m})) T\right)
\le c_3 M x \, 2^{-m} .
\end{multline*}

Using this and (\ref{eq:bmtk}), one obtains
\begin{multline*}
\left| b_m(t_k, x) - r_m^{-k-1} \left(f_m(t_{k+1}, u_m x) -  u_m x \, \partial_x f_m(t_{k+1}, x) \right) \right| \\
\le u_m x \left|a_m(t_k, x) -  \partial_x f_m(t_{k+1}, x) \right|
\le c_3 M (1 + \sigma 2^{-m} +|\mu| 2^{-2m}) x^2 \, 2^{-m}.
\end{multline*}
These prove (a).

(b) First, by (\ref{eq:at}),
\begin{multline}\label{eq:ama}
\sup_{0\le t \le T} |a_m(t^{(m)}) - a(t)| \le \sup_{0\le t \le T} \left|a_m(t^{(m)}) - \partial_x f_m(t^{(m)} + 2^{-2m}, S_m(t^{(m)}))\right| \\
+ \sup_{0\le t \le T} \left|\partial_x f_m(t^{(m)} + 2^{-2m}, S_m(t^{(m)})) - \partial_x f(t, S(t))\right| .
\end{multline}
By part (a) of this lemma and by (\ref{eq:Smconv}), the first term on the right hand side converges to $0$ almost surely as $m \to \infty$. (For any $\omega \in \Omega$ fixed, the function $t \mapsto S(\omega, t)$ has compact range over $[0, T]$.)

By Lemma \ref{le:diff_fm} and the Markov property of $S_m$, we have
\begin{multline}\label{eq:dfmSm}
\partial_x f_m(t^{(m)} + 2^{-2m}, x) = r_m^{\lfloor t 2^{2m}\rfloor + 1 - \lfloor T 2^{2m} \rfloor} \\
\times \mathbb{E}^x_{\mathbb{Q}_m} \left\{\frac{S_m(T^{(m)} - t^{(m)} - 2^{-2m})}{x} g'(S_m(T^{(m)} - t^{(m)} - 2^{-2m}))\right\} .
\end{multline}
By (\ref{eq:StQ}) we can get a similar formula for $\partial_x f(t, x)$ as well:
\begin{multline}\label{eq:dfS}
\partial_x f(t, x) =  e^{r(t-T)} \partial_x \mathbb{E}^x_{\mathbb{Q}} \left\{g(S(T-t))\right\} \\
= e^{r(t-T)} \partial_x  \mathbb{E}_{\mathbb{Q}} \left\{g\left(x e^{\left(r - \sigma^2/2 \right)(T-t) + \sigma W(T-t)}\right)\right\} \\
= e^{r(t-T)} \mathbb{E}^x_{\mathbb{Q}} \left\{\frac{S(T-t)}{x} g'(S(T-t))\right\} .
\end{multline}
Then replacing $x$ by $S_m(t^{(m)})$ in (\ref{eq:dfmSm}) and by $S(t)$ in (\ref{eq:dfS}), and using the assumption that $g \in C_c^2(\mathbb{R}_+)$ , Lemma \ref{le:Qmprop}(b), and formulas (\ref{eq:Smconv}) and (\ref{eq:betamconv}), it follows that the second term on the right hand side of (\ref{eq:ama}) converges to $0$. (For any $\omega \in \Omega$ fixed, $\inf_{0 \le t \le T} S(\omega, t) > 0$.)

By (\ref{eq:bt}) and (\ref{eq:bmtk}), the a.s. uniform convergence of $b_m(t^{(m)})$ to $b(t)$ follows in a similar manner. This completes the proof of (b).
\end{proof}

We mention that the convergence of $(\partial_x)^j f_m(t^{(m)}, x)$ to $(\partial_x)^j f(t, x)$  as $m \to \infty$ can be proved similarly for $j > 1$ (when $g \in C^j_c(\mathbb{R}_+)$) as for $j=1$ above.

\section{Replicating portfolio in the case of European options} \label{sec:Port_Eur}

Now we would like to extend Theorem \ref{th:appr_port} to the case of European put and call options.
The differentiability properties of the discrete European options $P_m(t_k, x)$ and $C_m(t_k, x)$ are worse than the ones we used above. The problem is caused by the fact that the functions $(s-K)_+$ and $(K-s)_+$ are not differentiable at the point $s=K$. That is why we are going to introduce smooth approximations. Because of the put-call parity (\ref{eq:dparity}), it is enough to consider the function $g(s) = (K-s)_+$ determining the put option, which has a compact support for $s \in \mathbb{R}_+$: this is the function that we consider from now on.

A standard technique is to use convolution with a smooth function of compact support:
\[
g_n(s) := \int_{-\infty}^{\infty} g(s-u) \, n \, \psi(nu) \di u  \quad (n = 1, 2, \dots),
\]
where
\begin{equation}\label{eq:psi}
\psi(u) := \left\{\begin{array}{ccc}
                    c \exp\left(-\frac{1}{1-u^2}\right) & \text{if} & |u| < 1, \\
                    0 & \text{if} & |u| \ge 1.
                  \end{array}
           \right.
\end{equation}
The constant $c$ is chosen so that $\int_{-\infty}^{\infty} \psi(u) \di u = 1$, $c \approx 2.25228$.  Then
\begin{multline*}
g_n(s) = \int_{\max\{s-K , -\frac{1}{n}\}}^{\frac{1}{n}} (K-s+u) \, n \, \psi(nu) \di u \\
= \left\{\begin{array}{ccc}
K-s & \text{if} & s \le K-\frac{1}{n} , \\
\int_{n(s-K)}^{1} \left(K-s+\frac{v}{n}\right) \psi(v)  \di v &\text{if} & |s-K| < \frac{1}{n}, \\
0 & \text{if} & s \ge K+\frac{1}{n} ,
 \end{array}
\right.
\end{multline*}
$g_n \in C^{\infty}(\mathbb{R})$. Consequently,
\begin{equation}\label{eq:gnprime}
g'_n(s) = \left\{\begin{array}{ccc}
-1 & \text{if} & s \le K-\frac{1}{n} , \\
- \int_{n(s-K)}^{1} \psi(v)  \di v &\text{if} & |s-K| < \frac{1}{n}, \\
0 & \text{if} & s \ge K+\frac{1}{n} , \end{array}
\right.
\end{equation}
and
\begin{equation}\label{eq:gndprime}
g''_n(s) = \left\{\begin{array}{ccc}
n \psi(n(s-K))  &\text{if} & |s-K| < \frac{1}{n}, \\
0 & \text{otherwise.}  &  \end{array}
\right.
\end{equation}

We will need an upper bound for the probability of an event that plays an important role in our smooth approximation.
\begin{lem}\label{le:QmSmK}
For any $T-t_k \ge \delta > 0$, $x > 0$, and $m \ge m_0$, $n \ge 1$ such that $\frac{1}{K} \left(\frac{1}{n} + c_1 2^{-m} \right) \le \frac12$, we have
\[
\mathbb{Q}_m^x\left(\left| S_m(T-t_k) - K \right| \le \frac{1}{n} \right)
\le \frac{c_5}{\sqrt{\delta}} \frac{1}{n} + \frac{c_6}{\sqrt{\delta}} 2^{-m} ,
\]
where $c_1 \in \mathbb{R}_+$ was defined in Lemma \ref{le:disc_appr} and $c_5= c_5(\mu,\sigma,r, T, K)$, $c_6 = c_6(\mu,\sigma,r, T, K)$ are positive constants.
\end{lem}
\begin{proof}
By Lemma \ref{le:disc_appr}, we may approximate $S_m$ by $\widetilde{S}_m$:
\begin{multline}\label{eq:QmSm}
\mathbb{Q}_m^x\left(\left| S_m(T-t_k) - K \right| \le \frac{1}{n} \right)
\le \mathbb{Q}_m^x\left(\left| \widetilde{S}_m(T-t_k) - K \right| \le \frac{1}{n} + c_1 2^{-m} \right) \\
= \mathbb{Q}_m\left(\left| x e^{\sigma W_m(T-t_k)+(r-\sigma^2/2)(T-t_k)} - K \right| \le \frac{1}{n} + c_1 2^{-m} \right) \\
= \mathbb{Q}_m\left(a^- \le \frac{W_m(T-t_k)}{\sqrt{T-t_k}} \le a^+\right),
\end{multline}
where $t_k < T$ and
\begin{equation}\label{eq:apm}
a^{\pm} = \frac{1}{\sigma \sqrt{T-t_k}} \left(\log\left(\frac{K \pm \left(\frac{1}{n} + c_3 2^{-m}\right)}{x}\right) - \left(r-\frac{\sigma^2}{2}\right)(T-t_k)\right).
\end{equation}

Now we want to estimate the difference
\[
\left| \mathbb{Q}_m\left(a^- \le \frac{W_m(T-t_k)}{\sqrt{T-t_k}} \le a^+\right) - \mathbb{Q}\left(a^- \le \frac{W(T-t_k)}{\sqrt{T-t_k}} \le a^+\right) \right|  ,
\]
using the Berry--Esseen theorem. The second probability above is simply $\Phi(a^+) - \Phi(a^-)$, because $W$ is $\mathbb{Q}$-Brownian motion. The first probability above depends on
\begin{multline*}
W_m(T-t_k) = B_m(T-t_k) + \frac{\mu-r}{\sigma} (T-t_k)  \\
= \sum_{i=1}^{N-k} (X_m(i) + \theta_m) 2^{-m} =: \sum_{i=1}^{N-k} \eta_m(i),
\end{multline*}
where $(X_m(i))_{i \ge 1}$ is a sequence of i.i.d. random variables, $\mathbb{Q}_m(X_m(i)=\pm 1) = q_m^{\pm}$, $\theta_m := \frac{\mu-r}{\sigma} 2^{-m}$, and $(\eta_m(i))_{i \ge 1}$ is also a sequence of i.i.d. random variables.

Then for any $m > m_0$ and $T-t_k \ge \delta >0$,
\begin{equation}\label{eq:BerryE}
\left| \mathbb{Q}_m\left(a^- \le \frac{W_m(T-t_k)}{\sqrt{T-t_k}} \le a^+\right) - \left(\Phi(a^+) - \Phi(a^-)\right) \right| \le c_4 2^{-m} ,
\end{equation}
where $c_4 = c_4(\mu, \sigma, r, T) \in \mathbb{R}_+$.

For, it is easy to check that $\mathbb{E}_{\mathbb{Q}_m}(\eta_m(i)) = 0$, $\mathbb{E}_{\mathbb{Q}_m}(|\eta_m(i)|^2) = 2^{-2m} (1 - \theta^2_m)$, and $\mathbb{E}_{\mathbb{Q}_m}(|\eta_m(i)|^3) = 2^{-3m} (1 - \theta^4_m)$. Hence by the Berry--Esseen theorem,
\begin{multline*}
\left| \mathbb{Q}_m\left(a^- \le \frac{W_m(T-t_k)}{\sqrt{T-t_k}} \le a^+\right) - \left(\Phi(a^+) - \Phi(a^-)\right) \right| \\
< \frac{2^{-3m} (1 - \theta^4_m)}{(2^{-2m} (1 - \theta^2_m))^{\frac32} \sqrt{(T-t_k)2^{2m}}} \le c_4 2^{-m}, \end{multline*}
with some positive constant $c_4 = c_4(\mu, \sigma, r, T)$.

Further, using the upper estimate
\[
\log(1+x) - \log(1-x) = 2 \left(x + \frac{x^3}{3} + \frac{x^5}{5} + \cdots \right) \le \frac{8}{3} x \quad \text{if} \quad |x| \le \frac12,
\]
by (\ref{eq:apm}) it follows that
\[
a^+ - a^- \le \frac{8 \left(\frac{1}{n} + c_1 2^{-m} \right)}{3 K \sigma \sqrt{T-t_k}}  \quad \text{if} \quad \frac{1}{K} \left(\frac{1}{n} + c_1 2^{-m} \right) \le \frac12,
\]
and then
\begin{equation}\label{eq:QmWm}
\Phi(a^+) - \Phi(a^-) = \phi(u) (a^+ - a^-) \\
\le \frac{8\left(\frac{1}{n} + c_1 2^{-m}\right)}{3 K \sigma \sqrt{2 \pi \delta}}  ,
\end{equation}
where $u \in (a^-, a^+)$, $T-t_k \ge \delta > 0$. Formulas (\ref{eq:QmSm}), (\ref{eq:BerryE}), and (\ref{eq:QmWm}) imply the statement of the lemma.
\end{proof}

Let us introduce a smooth approximation of the price $P_m(t_k, x)$ of the discrete European put option by
\begin{equation}\label{eq:Pmn}
P_m^{(n)}(t_k, x) := r_m^{k-N} \mathbb{E}^x_{\mathbb{Q}_m} \left\{g_n(S_m(T-t_k)) \right\} ,
\end{equation}
assuming that $T=N \Delta t$. (Here we used the Markov property of $S_m$.)
\begin{lem}\label{le:PmPmn}
For any $T-t_k \ge \delta > 0$, $x > 0$, and $m \ge m_0$, $n \ge 1$ such that $\frac{1}{K} \left(\frac{1}{n} + c_1 2^{-m} \right) \le \frac12$, we have
\[
\left|P_m^{(n)}(t_k, x) - P_m(t_k, x)\right| \le \frac{1}{2n} \left(\frac{c_5}{\sqrt{\delta}} \frac{1}{n} + \frac{c_6}{\sqrt{\delta}} 2^{-m} \right),
\]
where $c_1 \in \mathbb{R}_+$ was defined in Lemma \ref{le:disc_appr} and $c_5, c_6 \in\mathbb{R}_+$ were defined in Lemma \ref{le:QmSmK}.
\end{lem}
\begin{proof}
\begin{multline*}
\left|P_m^{(n)}(t_k, x) - P_m(t_k, x)\right| \le r_m^{k-N} \mathbb{E}^x_{\mathbb{Q}_m} \left\{ \left|g_n(S_m(T-t_k)) - g(S_m(T-t_k))\right| \right\} \\
\le r_m^{k-N} \mathbb{Q}_m \left\{\left| S_m(T-t_k) - K \right| \le \frac{1}{n} \right\} \frac{1}{2n} \le \frac{1}{2n} \left(\frac{c_5}{\sqrt{\delta}} \frac{1}{n} + \frac{c_6}{\sqrt{\delta}} 2^{-m} \right)
\end{multline*}
by Lemma \ref{le:QmSmK}.
\end{proof}

Each $g_n$ is a continuous function with support contained in a common compact interval $[0, K+1]$ for $s \ge 0$ and the sequence $(g_n)$ converges to $g(s)=(K-s)_+$ as $n \to \infty$. Hence taking e.g. $n = m$, similarly to the proof of Theorem \ref{th:slimit}, we obtain that
\begin{multline}\label{eq:Pmconv}
\lim_{m \to \infty} P_m^{(m)}(t^{(m)}, x)
= \lim_{m \to \infty} \mathbb{E}^x_{\mathbb{Q}_m}\left(r_m^{\lfloor t2^{2m} \rfloor - \lfloor T2^{2m}\rfloor} g_{m}(S_m(T^{(m)}- t^{(m)})) \right) \\
=  \mathbb{E}^x_{\mathbb{Q}}\left(e^{-r(T-t)}(K - S(T-t))_+  \right) = P(t,x) ,
\end{multline}
uniformly for $x > 0$ and $t \in [0, T-\delta]$, where $\delta > 0$ is arbitrary.

Consider now the convergence of the first $x$-derivative of the smooth approximation, assuming that $T = N \Delta t$ and defining $g'(K) = -\frac12$.
\begin{lem}\label{le:dPmPmn}
For any $T-t_k \ge \delta > 0$, $x > 0$, and $m \ge m_0$, $n \ge 1$ such that $\frac{1}{K} \left(\frac{1}{n} + c_1 2^{-m} \right) \le \frac12$, we have
\[
| \partial_x P_m^{(n)}(t_k, x) - \partial_x P_m(t_k, x) |
\le  \frac{K+\frac{1}{n}}{2x} \left(\frac{c_5}{\sqrt{\delta}} \frac{1}{n} + \frac{c_6}{\sqrt{\delta}} 2^{-m} \right),
\]
where $c_1 \in \mathbb{R}_+$ was defined in Lemma \ref{le:disc_appr} and $c_5, c_6 \in\mathbb{R}_+$ were defined in Lemma \ref{le:QmSmK}.
\end{lem}
\begin{proof}
Using Lemmas \ref{le:diff_fm} and \ref{le:QmSmK}, plus the fact that $S_m$ is Markov, we obtain
\begin{multline*}
| \partial_x P_m^{(n)}(t_k, x) - \partial_x P_m(t_k, x) | \\
\le  r_m^{k-N} \mathbb{E}^x_{\mathbb{Q}_m} \left\{ \frac{S_m(T-t_k)}{x} \left| g'_n(S_m(T-t_k)) - g'(S_m(T-t_k))\right| \right\}\\
\le r_m^{k-N} \mathbb{Q}_m^x\left(|S_m(T-t_k) - K| \le \frac{1}{n} \right)
\max_{ |s - K| \le \frac{1}{n}} \, \frac{s}{2x}  \\
\le \frac{K+\frac{1}{n}}{2x} \, \left(\frac{c_5}{\sqrt{\delta}} \frac{1}{n} + \frac{c_6}{\sqrt{\delta}} 2^{-m} \right).
\end{multline*}
\end{proof}

Again, taking e.g. $n=m$ and using the fact that  $|g'(s)| \le 1$ for any $s \ge 0$, similarly to the proof of Theorem \ref{th:slimit} we see that
\begin{multline}\label{eq:dxPmPconv}
\lim_{m \to \infty} \partial_x P_m(t^{(m)}, x) \\
= \lim_{m \to \infty} r_m^{\lfloor t 2^{2m} \rfloor - \lfloor T 2^{2m} \rfloor} \mathbb{E}^x_{\mathbb{Q}_m}\left(\frac{S_m(T^{(m)} - t^{(m)})}{x} g'(S_m(T^{(m)} - t^{(m)})) \right) \\
= e^{-r(T-t)} \mathbb{E}^x_{\mathbb{Q}}\left(\frac{S(T-t)}{x} g'(S(T-t)) \right) = \partial_x P(t, x),
\end{multline}
converging uniformly for $x \ge \delta$ and $t \in [0, T-\delta]$, where $\delta > 0$ is arbitrary. Here we used Lemma \ref{le:diff_f} as well.

Now we can extend Theorem \ref{th:appr_port} to the case of European put option, consequently, to the case of European call option as well.
\begin{thm}\label{th:appr_port_eur}
Consider the replicating portfolio $(a_m(t_k, x), b_m(t_k, x))$ given by (\ref{eq:amtk}) and (\ref{eq:bmtk}) in the case of European put option, where $g(s) = (K - s)_+ $. Then we can use the smooth approximation (\ref{eq:Pmn}) of the price $P_m(t_k, x)$ of the option to find an approximate replicating portfolio.

(a) For any $T-t_k \ge \delta > 0$, $x > 0$, and $m \ge m_0$, $n=\lceil 2^{2m/3} \rceil$ such that $\frac{1}{K} \left(\frac{1}{n} + c_1 2^{-m} \right) \le \frac12$, with the positive constants $c_1, c_3, c_5, c_6$ defined in Lemma \ref{le:disc_appr}, Theorem \ref{th:appr_port}, and Lemma \ref{le:QmSmK}, respectively, we have
\[
\left| a_m(t_k, x) - \partial_x P_m^{(n)}(t_{k+1}, x)\right| \le \left( 2 c_3 x + \frac{c_5  + c_6}{2 \sigma x \sqrt{\delta}} \right) 2^{-m/3}
\]
and
\begin{multline*}
\left| b_m(t_k, x) - r_m^{-k-1} \left(P_m(t_{k+1}, u_m x) -  u_m x \, \partial_x P_m^{(n)}(t_{k+1}, x) \right) \right| \\
\le (1 + \sigma 2^{-m} + |\mu| 2^{-2m}) \left(2 c_3 x^2 + \frac{c_5  + c_6}{2 \sigma \sqrt{\delta}} \right) 2^{-m/3}.
\end{multline*}

(b) Moreover, using the notation $t^{(m)} = \lfloor t 2^{2m} \rfloor 2^{-2m}$,
\[
\lim_{m \to \infty} \sup_{0\le t \le T} |a_m(t^{(m)}) - a(t)| = 0 , \quad \lim_{m \to \infty} \sup_{0\le t \le T} |b_m(t^{(m)}) - b(t)| = 0 \quad \text{a.s.},
\]
where the portfolio $(a(t), b(t))$ of the Black--Scholes model is defined by (\ref{eq:at}) and (\ref{eq:bt}).
\end{thm}
\begin{proof}
(a) By (\ref{eq:amtk}),
\[
a_m(t_k, x) = \frac{P_m(t_{k+1}, u_m x - P_m(t_{k+1}, d_m x)}{(u_m - d_m) x} .
\]
Define
\[
a_m^{(n)}(t_k, x) := \frac{P_m^{(n)}(t_{k+1}, u_m x - P_m^{(n)}(t_{k+1}, d_m x)}{(u_m - d_m) x} .
\]
Then
\begin{multline*}
\left|a_m(t_k, x) - a_m^{(n)}(t_k, x) \right| \\
\le \frac{\left|P_m(t_{k+1}, u_m x) - P_m^{(n)}(t_{k+1}, u_m x) \right| + \left|P_m(t_{k+1}, d_m x) - P_m^{(n)}(t_{k+1}, d_m x) \right|}{2 \sigma x 2^{-m}} \\
\le \frac{2^m}{2 \sigma x n} \left(\frac{c_5}{\sqrt{\delta}} \frac{1}{n} + \frac{c_6}{\sqrt{\delta}} 2^{-m} \right) = \frac{1}{2 \sigma x } \left(\frac{c_5}{\sqrt{\delta}} \frac{2^{m}}{n^2} + \frac{c_6}{\sqrt{\delta}} \frac{1}{n} \right) ,
\end{multline*}
by Lemma \ref{le:dPmPmn}.

Now let us extend the method of Theorem \ref{th:appr_port}(b) to $a_m^{(n)}(t_k, x)$. The only difference is that $M_n = \sup_{s \ge 0} |g''_n(s)|$ depends on $n$, unlike (\ref{eq:Maxgdprime}). However, by (\ref{eq:psi}) and (\ref{eq:gndprime}), $M_n = n c/e < n$. Thus we get that
\[
\left| a_m^{(n)}(t_k, x) - \partial_x P_m^{(n)}(t_{k+1}, x)\right| \le c_3 n x \, 2^{-m} ,
\]
where $c_3$ is the same as in Theorem \ref{th:appr_port}(b). In sum,
\[
\left| a_m(t_k, x) - \partial_x P_m^{(n)}(t_{k+1}, x)\right| \le \frac{1}{2 \sigma x } \left(\frac{c_5}{\sqrt{\delta}} \frac{2^{m}}{n^2} + \frac{c_6}{\sqrt{\delta}} \frac{1}{n} \right) + c_3 n x \, 2^{-m} .
\]
Let us choose e.g. $n=\lceil 2^{2m/3} \rceil$ here, then
\[
\left| a_m(t_k, x) - \partial_x P_m^{(n)}(t_{k+1}, x)\right| \le  \left(\frac{c_5  + c_6}{2 \sigma x \sqrt{\delta}}  + 2 c_3 x \right) 2^{-m/3} .
\]

Further, by this result and by equation (\ref{eq:bmtk}),
\begin{multline*}
\left| b_m(t_k, x) - r_m^{-k-1} \left(P_m(t_{k+1}, u_m x) - u_m x \, \partial_x P_m^{(n)}(t_{k+1}, x) \right) \right| \\
\le r_m^{-k-1} u_m x \left| a_m(t_k, x) - \partial_x P_m^{(n)}(t_{k+1}, x)\right| \\
\le (1 + \sigma 2^{-m} + |\mu| 2^{-2m}) \left(2 c_3 x^2 + \frac{c_5  + c_6}{2 \sigma \sqrt{\delta}} \right) 2^{-m/3} .
\end{multline*}
These prove (a).

(b) These statements follow similarly from (a) as Theorem \ref{th:appr_port}(b) followed from Theorem \ref{th:appr_port}(a), therefore the details are omitted.
\end{proof}

Lemmas \ref{le:PmPmn}, \ref{le:dPmPmn} and Theorem \ref{th:appr_port_eur} extend to the case of European call option as well. For, by the discrete put-call parity (\ref{eq:dparity}), one can define
\[
C_m^{(n)}(t_k, x) := P_m^{(n)}(t_k, x) + x - r_m^{k-N} K, \quad \partial_x C_m^{(n)}(t_k, x) = \partial_x P_m^{(n)}(t_k, x) + 1 .
\]
Then the corresponding statements for the call option follow easily.

\section{Discrete Feynman--Kac formulas and a discrete Black--Scholes equation} \label{sec:dFK_BS}

We would like to have a discrete version of the celebrated Black--Scholes equation, and to obtain the latter as a limit of the discrete formula. In that program a useful tool is a \emph{discrete Feynman--Kac formula}. We mention that a similar formula was introduced by Cs\'aki \cite{Csa1993} in 1993. Let us start with a \emph{discrete time-homogeneous It\^o process}:
\begin{equation}\label{eq:dGBM}
\Delta S_m(t_{k+1}) = \mu(S_m(t_k)) \Delta t + \sigma(S_m(t_k)) \Delta B_m(t_{k+1}), \quad S_m(t_0) = s_0 \in \mathbb{R} ,
\end{equation}
where $\Delta S_m(t_{k+1}) := S_m(t_{k+1}) - S_m(t_k)$, $\Delta B_m(t_{k+1}) := B_m(t_{k+1}) - B_m(t_k) = \pm \Delta x$, and $\mu : \mathbb{R} \to \mathbb{R}$, $\sigma : \mathbb{R} \to \mathbb{R_+}$ are given continuous functions. Introduce a ``forward'' version of a \emph{discrete Feynman--Kac functional} for $m > m_0$ fixed:
\begin{equation}\label{eq:dFK1}
f_m(t_k, x) := \mathbb{E}^{x} \left\{\exp\left(-\sum_{i=0}^{k-1} r(S_m(t_i)) \Delta t \right) g(S_m(t_k)) \right\}  \quad (t_k \ge 0, x \in \mathbb{R}),
\end{equation}
where $\mathbb{E}^x(Z) := \mathbb{E}(Z | B(0) = x)$ and $r : \mathbb{R} \to \mathbb{R}$ and $g : \mathbb{R} \to \mathbb{R}$ are given continuous functions as well. Let us use the abbreviation
\begin{equation}\label{eq:spm}
x^{\pm} := x + \mu(x) \Delta t + \sigma(x) (\pm \Delta x) .
\end{equation}

\begin{lem}\label{le:dFK1}
For any $m > m_0$ fixed, the discrete Feynman--Kac functional (\ref{eq:dFK1}) is the unique solution of the difference equation (\emph{a discrete Feynman--Kac formula})
\begin{multline*}
\frac{f_m(t_{k+1},x) - f_m(t_k, x)}{\Delta t} = \frac12 e^{-r(x) \Delta t} \, \frac{f_m(t_k, x^+) + f_m(t_k, x^-) - 2 f_m(t_k, x)}{(\Delta x)^2} \\
+ \frac{e^{-r(x) \Delta t} - 1}{\Delta t} f_m(t_k, x),
\end{multline*}
with the initial condition $f_m(0, x) = g(x)$.
\end{lem}
\begin{proof}
Conditioning on the first step of $B_m$ and using the fact that $S_m$ is a time-homogeneous Markov process, one obtains
\begin{multline}\label{eq:steps2}
f_m(t_{k+1}, x) = e^{-r(x) \Delta t} \sum_{\delta = \pm \Delta x} \mathbb{P}(B_m(t_1) = \delta )  \\
\times  \mathbb{E}^{x} \left\{\exp\left(-\sum_{i=0}^{k-1} r(S_m(t_{i+1})) \Delta t \right) g(S_m(t_{k+1})) \mid B_m(t_1) = \delta \right\} \\
= \frac12 e^{-r(x) \Delta t} \sum_{j=\pm} \mathbb{E}^{x^{j}} \left\{\exp\left(-\sum_{i=0}^{k-1} r(S_m(t_{i})) \Delta t \right) g(S_m(t_{k})) \right\} \\
= \frac12 e^{-r(x) \Delta t} \left( f_m(t_k, x^+) + f_m(t_k, x^-) \right) .
\end{multline}
Now subtracting $f_m(t_k, x)$, dividing by $\Delta t$, and arranging the terms on the right hand side, one can get the difference equation for $f_m$ in the statement of the lemma.

The uniqueness follows by induction, starting with the initial condition and proceeding with $\Delta t$ steps in time by formula (\ref{eq:steps2}).
\end{proof}

To apply a discrete Feynman--Kac formula to the strong discrete financial model introduced in the previous section, a ``backward'' version of a discrete Feynman--Kac functional is needed. Such a ``backward'' discrete formula can similarly be shown as the ``forward'' version in Lemma \ref{le:dFK1}, even for explicitly time dependent discrete It\^o processes:
\begin{equation}\label{eq:dGtIP}
\Delta S_m(t_{k+1}) = \mu(t_k, S_m(t_k)) \Delta t + \sigma(t_k, S_m(t_k)) \Delta W_m(t_{k+1}),
\end{equation}
$S_m(t_0) = s_0 \in \mathbb{R}$, where $\mu : \mathbb{R}_+ \times \mathbb{R} \to \mathbb{R}$, $\sigma : \mathbb{R}_+ \times \mathbb{R} \to \mathbb{R_+}$ are given continuous functions, and $W_m$ is a simple random walk, not necessarily symmetric. Let $(\Delta x)^+$ denote an up-step and $(\Delta x)^-$ a down-step of $W_m$, and $p_m^+$ and $p_m^-$ denote their probabilities, respectively.

Consider the following functional when $T=N \Delta t$:
\begin{equation}\label{eq:dFKb}
f_m(t_k, x) := \mathbb{E} \left\{\left. \exp\left(-\sum_{i=k+1}^{N} \rho_m(t_i, S_m(t_i)) \Delta t \right) g(S_m(T)) \right| S_m(t_k) = x \right\}
\end{equation}
$(t_k \ge 0, x \in \mathbb{R})$, where $S_m$ is defined by (\ref{eq:dGtIP}), $\rho_m : \mathbb{R}_+ \times \mathbb{R} \to \mathbb{R}$ and $g : \mathbb{R} \to \mathbb{R}$ are continuous functions. Use the abbreviations
\[
x^{\pm} := x + \mu(t_k, x) \Delta t  \pm  \sigma(t_k, x) (\Delta x)^{\pm} .
\]
\begin{lem}\label{le:dFKback}
For any $m > m_0$ fixed, the discrete Feynman--Kac functional (\ref{eq:dFKb}) is the unique solution of the difference equation
\begin{multline*}
\frac{f_m(t_{k},x) - f_m(t_{k-1}, x)}{\Delta t} \\
+ \frac12 e^{-\rho_m(t_k, x) \Delta t} \, \frac{2 p_m^+ \, f_m(t_{k}, x^+) + 2 p_m^- \, f_m(t_{k}, x^-) -  2f_m(t_{k}, x)}{(\Delta x)^2} \\
+ \frac{e^{-\rho_m(t_k, x) \Delta t} - 1}{\Delta t} f_m(t_{k}, x) = 0,
\end{multline*}
with the boundary condition $f_m(T, x) = g(x)$.
\end{lem}
\begin{proof}
Conditioning on the value of $\Delta W_m(t_{k}) = W_m(t_k) - W_m(t_{k-1})$, we obtain
\begin{multline}\label{eq:stepsback}
f_m(t_{k-1}, x) = e^{-\rho_m(t_k, x) \Delta t} \sum_{\delta=(\Delta x)^{\pm}} \mathbb{P}(\Delta W_m(t_{k}) = \delta) \\
\times  \mathbb{E} \left\{ \left. \exp\left(-\sum_{i=k+1}^{N} \rho_m(t_i, S_m(t_i)) \Delta t \right) g(S_m(T)) \right| S_m(t_{k-1}) = x, \Delta W_m(t_{k}) = \delta  \right\} \\
= e^{-\rho_m(t_k, x) \Delta t} \\
\times \sum_{j=\pm} p_m^j \, \mathbb{E} \left\{ \left. \exp\left(-\sum_{i=k+1}^{N} \rho_m(t_i, S_m(t_i)) \Delta t \right) g(S_m(T)) \right| S_m(t_{k}) = x^{j} \right\} \\
= e^{-\rho_m(t_k, x) \Delta t} \left\{2 p_m^+ \, f_m(t_{k}, x^+) + 2 p_m^- \, f_m(t_{k}, x^-)\right\} .
\end{multline}
Now subtracting $f_m(t_{k}, x)$, dividing by $\Delta t$, and arranging the terms, one can get the difference equation for $f_m$ in the statement of the lemma.

The uniqueness follows by induction, starting with the boundary condition at time $T$ and proceeding with $-\Delta t$ steps backward in time by formula (\ref{eq:stepsback}).
\end{proof}

In the sequel we need a special case of Lemma \ref{le:dFKback} and we use the notations introduced in the previous Section \ref{sec:DBSM}:
\begin{equation}\label{eq:dFK2}
f_m(t_k, x) := \mathbb{E}_{\mathbb{Q}_m} \left\{\left. r_m^{k-N} g(S_m(T)) \right| S_m(t_k) = x \right\}  \end{equation}
$(t_k \ge 0, x \in \mathbb{R})$, where  $S_m$ is defined by (\ref{eq:Qmlinear}), and the continuous function $g : \mathbb{R} \to \mathbb{R}$ is used to describe a claim of the form $g(S_m(T))$ at maturity $T = N \Delta t$. Here $r_m = 1 + r \Delta t$, that is, $\rho_m = \log(r_m)/\Delta t$ in the previous Lemma \ref{le:dFKback}. Also, we now substitute
\begin{equation}\label{eq:spm2}
x^+ := x u_m ,  \qquad x^- := x d_m , \qquad p_m^{\pm} = q_m^{\pm}.
\end{equation}
\begin{lem}\label{le:dFK2}
For any $m > m_0$ fixed, the discrete Feynman--Kac functional (\ref{eq:dFK2}) is the unique solution of the difference equation
\begin{multline}\label{eq:FK_diffeq2}
\frac{f_m(t_{k},x) - f_m(t_{k-1}, x)}{\Delta t} \\
+ \frac12 r_m^{-1} \, \frac{2q_m^+ \, f_m(t_{k}, x^+) + 2q_m^- \, f_m(t_{k}, x^-) -  2f_m(t_{k}, x)}{(\Delta x)^2} \\
+ \frac{r_m^{-1} - 1}{\Delta t} f_m(t_{k}, x) = 0,
\end{multline}
with the boundary condition $f_m(T, x) = g(x)$.
\end{lem}

Here is a \emph{time discrete version of the Black--Scholes equation} based on the discrete It\^o process (\ref{eq:Qmlinear}) and functional (\ref{eq:dFK2}):
\begin{lem}\label{le:dBlackS}
Suppose that $g \in C^3_c(\mathbb{R})$. Then for any $m > m_0$ fixed, the discrete Feynman--Kac functional (\ref{eq:dFK2}) is an approximate solution of the difference-differential equation
\begin{multline}\label{eq:FK_diffeq3}
\frac{f_m(t_{k},x) - f_m(t_{k-1}, x)}{\Delta t}
+ r x \partial_x f_m(t_k, x) + \frac12 \sigma^2 x^2 \partial_{xx} f_m(t_k, x) - r f_m(t_k, x) \\
=  O(2^{-m}),
\end{multline}
with the boundary condition $f_m(T, x) = g(x)$.
\end{lem}
\begin{proof}
We apply a second order Taylor expansion with a third order error term in the $x$ variable to the second term in the equation (\ref{eq:FK_diffeq2}), using formulas (\ref{eq:um_dm}) and (\ref{eq:spm2}). First,
\begin{eqnarray*}
x^{\pm} - x &=& x \left(\pm \sigma 2^{-m} + \mu 2^{-2m} \right), \\
(x^{\pm} - x)^2 &=& x^2 \left(\sigma^2 2^{-2m} + O(2^{-3m}) \right), \\
(x^{\pm} - x)^3 &=& x^3 \left(\pm \sigma^3 2^{-3m} + O(2^{-4m}) \right) = O(2^{-3m}) .
\end{eqnarray*}

Second, if $M := \sup_{s > 0} |g'''(s)|$, then by (\ref{eq:jth_der}) we obtain that
\begin{multline*}
|\partial_{xxx}f_m(t_k, \xi)| \le M \sum_{i=0}^{N-k} \binom{N-k}{i} (q_m^+ \, u_m^3 )^i \, (q_m^-  \, d_m^3)^{N-k-i} \\
= M |q_m^+ \, u_m^3 + q_m^-  \, d_m^3|^{N-k}.
\end{multline*}
Since
\begin{multline*}
q_m^+ \, u_m^3 + q_m^-  \, d_m^3 \\
=  \left(\frac12 + \frac12 \frac{r-\mu}{\sigma}2^{-m} \right) \left( 1 + 3 \sigma 2^{-m} + 3(\mu+\sigma^2) 2^{-2m} + O(2^{-3m}) \right)  \\
+  \left(\frac12 - \frac12 \frac{r-\mu}{\sigma}2^{-m} \right) \left( 1 - 3 \sigma 2^{-m} + 3(\mu+\sigma^2) 2^{-2m} + O(2^{-3m}) \right) \\
=   1 + 3(r+\sigma^2)2^{-2m} + O(2^{-3m}),
\end{multline*}
it follows that
\begin{multline*}
\left|\partial_{xxx}f_m(t_k, \xi) (x^{\pm} - x)^3 \right| \\
\le M \left| 1 + 3(r+\sigma^2)2^{-2m} + O(2^{-3m})\right|^{(T-t_k)2^{2m}} O(2^{-3m}) \\
\le M e^{\left(3(r+\sigma^2) + O(2^{-m})\right)T} O(2^{-3m}) = O(2^{-3m}).
\end{multline*}

Third, by the above-mentioned Taylor approximation,
\begin{multline*}
2q_m^+ \, f_m(t_{k}, x^+) + 2q_m^- \, f_m(t_{k}, x^-) -  2f_m(t_{k}, x) \\
= \left(1 + \frac{r-\mu}{\sigma}2^{-m} \right)
\bigg{\{} f_m(t_k, x) + \partial_x f_m(t_k, x) \, x \left( \sigma 2^{-m} + \mu 2^{-2m}  \right)  \\
+  \frac12 \partial_{xx}f_m(t_k, x) \, x^2 \left(\sigma^2 2^{-2m} + O(2^{-3m}) \right) + O(2^{-3m}) \bigg{\}} \\
+ \left(1 - \frac{r-\mu}{\sigma}2^{-m} \right)
\bigg{\{}f_m(t_k, x) + \partial_x f_m(t_k, x) \, x \left(- \sigma 2^{-m} + \mu 2^{-2m} \right) \\
+ \frac12 \partial_{xx}f_m(t_k, x) \, x^2 \left(\sigma^2 2^{-2m} + O(2^{-3m}) \right) + O(2^{-3m})\bigg{\}}
- 2f_m(t_{k}, x) \\
= \left\{ 2rx \, \partial_x f_m(t_k, x) + \sigma^2 x^2 \, \partial_{xx} f_m(t_k, x)\right\} 2^{-2m} + O(2^{-3m}).
\end{multline*}

Thus
\begin{multline*}
\frac12 r_m^{-1} \frac{2q_m^+ \, f_m(t_{k}, x^+) + 2q_m^- \, f_m(t_{k}, x^-) -  2f_m(t_{k}, x)}{(\Delta x)^2} \\
=  \left(1 + O(2^{-2m})\right)  \left\{ rx \, \partial_x f_m(t_k, x) + \frac12\sigma^2 x^2 \, \partial_{xx} f_m(t_k, x) + O(2^{-m}) \right\}.
\end{multline*}
Substituting this into (\ref{eq:FK_diffeq2}), together with $(r_m^{-1} - 1)/\Delta t = -r (1+O(2^{-2m}))$, we get the statement of the lemma.
\end{proof}

\begin{thm}\label{th:BS_lim}
Suppose that $g \in C_c^3(\mathbb{R})$. Then
\begin{multline}\label{eq:fm_lim}
\lim_{m \to \infty} f_m(t^{(m)}, x) = \lim_{m \to \infty} r_m^{\lfloor t2^{2m} \rfloor - \lfloor T2^{2m} \rfloor} \mathbb{E}^x_{\mathbb{Q}_m}\left( g(S_m(T^{(m)} - t^{(m)})) \right) \\
=  e^{-r(T-t)} \mathbb{E}^x_{\mathbb{Q}}\left( g(S(T-t))\right) = f(t, x) ,
\end{multline}
for any $0 \le t \le T$ and $x > 0$. This implies that the price function $f(t, x)$ of the claim $g(S(T))$ is a solution of the Black--Scholes partial differential equation
\begin{equation}\label{eq:PDE}
\partial_t f(t, x) + r x \partial_x f(t, x) + \frac12 \sigma^2 x^2 \partial_{xx} f(t, x) - r f(t, x) = 0,
\end{equation}
with boundary condition $f(T,x)=g(x)$, where $0 \le t \le T$ and $x > 0$ are arbitrary.
\end{thm}
\begin{proof}
The limit in (\ref{eq:fm_lim}) can be proved similarly as the limit in Theorem \ref{th:slimit}. The fact that $f(t, x)$ solves the equation (\ref{eq:PDE}) follows from (\ref{eq:FK_diffeq3}) replacing $t_k$ by $t^{(m)}$ and taking a limit of the right hand side as $m \to \infty$. This also proves that $f(t, x)$ is continuously differentiable with respect to $t$ as well and
\[
\lim_{m \to\infty} \frac{f_m(t^{(m)},x) - f_m(t^{(m)} - \Delta t, x)}{\Delta t} = \partial_t f(t, x) .
\]
\end{proof}

Lemma \ref{le:dBlackS} and Theorem \ref{th:BS_lim} can easily be extended to the case of European put and call options, replacing $f_m$ by the smoothed versions $P_m^{(n)}$ or $C_m^{(n)}$ of the price processes.


\end{document}